\documentclass[11pt]{amsart}

\usepackage[margin=1.05in]{geometry}
\usepackage{amsmath,amssymb,amsthm,mathtools}
\usepackage{microtype}
\usepackage[hidelinks]{hyperref}
\hypersetup{
  pdftitle={A finiteness result for split exceptional dual pairs},
  pdfauthor={Hung Yean Loke and Gordan Savin}
}

\newcommand{\C}{\mathbb C}
\newcommand{\g}{\mathfrak g}
\newcommand{\kalg}{\mathfrak k}
\newcommand{\p}{\mathfrak p}
\newcommand{\Sp}{\mathrm {Sp}}
\newcommand{\SL}{\mathrm {SL}}
\newcommand{\SU}{\mathrm {SU}}
\newcommand{\SO}{\mathrm {SO}}
\newcommand{\GL}{\mathrm {GL}}
\newcommand{\Spin}{\mathrm {Spin}}
\newcommand{\Gr}{\operatorname {Gr}}
\newcommand{\LG}{\operatorname {LG}}
\newcommand{\OGr}{\operatorname {OGr}}
\newcommand{\Sym}{\operatorname {Sym}}

\theoremstyle{plain}
\newtheorem{proposition}{Proposition}[section]
\newtheorem{theorem}[proposition]{Theorem}

\theoremstyle{remark}

\title{A finiteness result for split exceptional dual pairs}
\dedicatory{To Ranee Brylinski, in memoriam.}

\author{Hung Yean Loke}
\address{Department of Mathematics, National University of Singapore,
21 Lower Kent Ridge Road, Singapore 119077}
\email{matlhy@nus.edu.sg}

\author{Gordan Savin}
\address{Department of Mathematics, University of Utah,
Salt Lake City, Utah 84112}
\email{savin@math.utah.edu}

\subjclass[2010]{Primary 11F27, 22E46}
\keywords{Exceptional theta correspondences, minimal representations,
Harish--Chandra modules}
\thanks{This work is supported by the Croatian Science Foundation under
the project IP-2022-10-4615, a gift No.~946504 from the Simons Foundation
and an NUS MOE grant A-8002493-00-00.}

\begin{document}

\begin{abstract}
We prove that big theta lifts have finite length for thirteen split
exceptional dual pairs arising from the minimal representations of the
split real groups of types $E_6$, $E_7$, and $E_8$.
\end{abstract}

\maketitle

\section{Introduction}\label{sec:introduction}

Let $\mathfrak e$ be a split real Lie algebra of type $E_6$, $E_7$, or
$E_8$, and let $\Omega$ be the Harish--Chandra module of its spherical
minimal representation.  A reductive dual pair in $\mathfrak e$ will mean
a pair of reductive subalgebras
\[
                       \g_1\oplus\g_2\subset\mathfrak e
\]
which are preserved by a Cartan involution and are mutual centralizers, up
to the harmless center issues. Write
$K_1$ and $K_2$ for the corresponding maximal compact subgroups.  If
$\pi_2$ is an irreducible $(\g_2,K_2)$-module, its big theta lift is the
$(\g_1,K_1)$-module
\[
       \Theta(\pi_2)
          =(\Omega\otimes\pi_2^\vee)_{(\g_2,K_2)}.
\]
Our purpose is to prove that these algebraic big theta lifts have finite
length for a broad family of split exceptional dual pairs.

Each split algebra $\mathfrak e$, for $n=6,7,8$, contains a nested chain
of split subalgebras of types
\begin{equation}\label{eq:intro-chain}
                 A_1\subset A_2\subset G_2\subset D_4\subset F_4.
\end{equation}
Taking the centralizer of each member gives the pairs displayed in
\eqref{eq:centralizer-table}; see also \cite{LSR}.  There are five entries
for each of the three ambient algebras.  In $E_8$, however, the
centralizer of $G_2$ is $F_4$ and the centralizer of $F_4$ is $G_2$.
Those two entries therefore describe the same unordered pair, leaving
fourteen distinct pairs in all.

Our main result treats all but the symmetric $D_4$ pair in $E_8$.

\begin{theorem}\label{thm:main}
Let $\g_1\oplus\g_2\subset\mathfrak e$ be one of the fourteen split
centralizer pairs obtained from \eqref{eq:intro-chain}.  Unless
\[
                         \g_1\oplus\g_2
                           =D_4\oplus D_4\subset E_8,
\]
the big theta lift of every irreducible Harish--Chandra module has finite
length.  This holds for theta lifting in either direction.  Thus the
theorem applies to thirteen split exceptional dual pairs.
\end{theorem}

We describe the argument.  For an irreducible $K_2$-module $\tau$,
consider the compact-type lift
\[
             \Theta_1(\tau)
                 =(\Omega\otimes\tau^\vee)_{K_2}.
\]
The central point is that $\Theta_1(\tau)$ is finitely generated over
$U(\g_1)$.  Three ingredients enter the proof.

First, give $\Omega$ its natural $K$-invariant filtration and put
\[
                          R=\operatorname{gr}\Omega.
\]
Then $R$ is the coordinate ring of the highest-weight cone associated
with the minimal representation: these are the Pl\"ucker cones over
$\LG(4,8)$ and $\Gr(4,8)$ in types $E_6$ and $E_7$, and the pure-spinor
cone over $\OGr^+(8,16)$ in type $E_8$.  Second, the Hilbert basis theorem
and averaging over $K_2$ imply that the covariant space
\[
                        (R\otimes\tau^\vee)_{K_2}
\]
is a finitely generated module over $R^{K_2}$.  Third, with $\g_1$ chosen
as the smaller member of the dual pair, the natural graded map
\begin{equation}\label{eq:intro-invariant-map}
                         \Sym(\p_1)\longrightarrow R^{K_2}
\end{equation}
is an isomorphism in the thirteen cases of the theorem.  More precisely,
it arises from a strict filtered isomorphism between the universal
spherical $(\g_1,K_1)$-module and $\Omega^{K_2}$.  The injectivity of
the filtered action map follows from restriction estimates placing the
matrix coefficients of $\Omega$ in $L^{2-\varepsilon}(G_1)$; the
surjectivity then follows from the calculation of the graded invariant
dimensions.

The first two ingredients and \eqref{eq:intro-invariant-map} show that
the associated graded compact-type lift is finitely generated over
$\Sym(\p_1)$.  A filtered-generation argument then lifts its homogeneous
generators and proves finite generation over $U(\g_1)$.  Taking a
$K_2$-type in an irreducible module $\pi_2$ exhibits $\Theta(\pi_2)$ as a
quotient of such a compact-type lift.  Consequently, the big theta lift
is finitely generated in one direction.  Applying the same argument from
the other side, through the usual see-saw adjunction, proves admissibility
in the reverse direction.  The known correspondence of infinitesimal
characters, together with the standard finiteness theorems for
Harish--Chandra modules, then gives finite length in both directions.

The paper is organized as follows.  Section~\ref{sec:filtered-generation}
proves the elementary filtered-generation lemma.  Sections
\ref{sec:minimal-modules} and \ref{sec:cones} recall the three minimal
representations and their homogeneous cones.  Section
\ref{sec:compact-lifts} proves finite generation of compact-type lifts by
averaging over the compact group,
and Section~\ref{sec:big-theta} passes from these modules to big theta
lifts.  In Section~\ref{sec:invariant-cyclicity} we verify
the required invariant dimensions for the split centralizer pairs and
reduce Theorem~\ref{thm:main} to the injectivity of the filtered action
maps.  Section~\ref{sec:decay} proves the matrix-coefficient estimates
that supply this injectivity and thereby completes the argument.

This paper was written with assistance of ChatGTP. The organization of the paper 
and the proofs are due to the authors.  

\section{A filtered-generation lemma}\label{sec:filtered-generation}

Let $(\g,K)$ be a Harish--Chandra pair, with complexified Cartan
decomposition
\[
                       \g=\kalg\oplus\p.
\]
Let $V$ be a $(\g,K)$-module with an exhaustive increasing filtration
\[
 0=V_{-1}\subset V_0\subset V_1\subset\cdots,
 \qquad V=\bigcup_{n\geq0}V_n,
\]
such that every $V_n$ is $K$-stable and
\[
                         \p V_n\subset V_{n+1}.
\]
Equivalently, $\g V_n\subset V_{n+1}$; the sharper degree-zero statement
for $\kalg$ follows from $K$-stability.
The $K$-stability implies $\kalg V_n\subset V_n$.  Consequently $\p$ acts
in degree one on
\[
                    \operatorname{gr}V
                    =\bigoplus_{n\geq0}V_n/V_{n-1}.
\]
If $X,Y\in\p$, then $[X,Y]\in\kalg$ preserves the filtration degree.
Hence $XY-YX$ has degree at most one less than $XY$ on the associated
graded module.  The degree-one operators coming from $\p$ therefore
commute, and $\operatorname{gr}V$ is naturally a graded
$(\Sym(\p),K)$-module.

\begin{proposition}[Filtered generation]\label{prop:filtered-generation}
Suppose that $\operatorname{gr}V$ is generated over $\Sym(\p)$ by a
finite-dimensional homogeneous $K$-submodule
\[
                         F=\bigoplus_d F_d,
             \qquad F_d\subset V_d/V_{d-1}.
\]
For every $d$, choose a $K$-stable lift
$\widetilde F_d\subset V_d$ of $F_d$, and put
$\widetilde F=\bigoplus_d\widetilde F_d$.  Then
\[
                              V=U(\g)\widetilde F.
\]
Equivalently, there is a natural surjection
\[
 U(\g)\otimes_{U(\kalg)}\widetilde F\twoheadrightarrow V.
\]
Thus, if the associated graded module is generated by finitely many
homogeneous $K$-types, the original Harish--Chandra module is generated by
any $K$-stable filtered lifts of those occurrences.
\end{proposition}

\begin{proof}
The required $K$-stable lifts exist because finite-dimensional
$K$-modules are semisimple.  Let $W=U(\g)\widetilde F$.  We prove
$V_n\subset W$ by induction on $n$.  If $v\in V_n$, its leading term
$\sigma_n(v)\in V_n/V_{n-1}$ can be written
\[
       \sigma_n(v)=\sum_j f_j\,u_j,
       \qquad u_j\in F_{d_j},\quad
       f_j\in\Sym^{n-d_j}(\p).
\]
Choose $\widetilde u_j\in\widetilde F_{d_j}$ above $u_j$ and PBW lifts
$\widetilde f_j\in U(\g)$ of the polynomials $f_j$.  The leading term of
$\sum_j\widetilde f_j\widetilde u_j$ is $\sigma_n(v)$, and therefore
\[
             v-\sum_j\widetilde f_j\widetilde u_j\in V_{n-1}.
\]
The difference belongs to $W$ by induction, so $v\in W$.  This proves the
first assertion.  The second is the corresponding multiplication map,
which is well defined because $\widetilde F$ is $K$-stable.
\end{proof}

\section{The three split groups of type \texorpdfstring{$E$}{E}}
\label{sec:minimal-modules}

Let $H_i$ be the split real group of type $E_i$, for $i=6,7,8$, and let
$\Omega_i$ be its spherical minimal Harish--Chandra module.  It is the
irreducible spherical module whose annihilator is the Joseph ideal.  The
construction of Brylinski--Kostant \cite{BK} gives a uniform description
of its $K_i$-types.

Up to the indicated finite central quotients, the complexified Cartan
decompositions and the noncompact modules are as follows.  We number the
fundamental weights in the standard Bourbaki convention.
\[
\begin{array}{c|c|c|c}
H_i & K_i & \p_i & \text{highest weight of }\p_i \\ \hline
E_{6(6)} & \Sp(4)/\{\pm1\}
 & \Lambda^4_0\C^8 & \omega_4\ (C_4)\\
E_{7(7)} & \SU(8)/\{\pm I\}
 & \Lambda^4\C^8 & \omega_4\ (A_7)\\
E_{8(8)} & \Spin(16)/\mu_2
 & S^+ & \omega_8\ (D_8)
\end{array}
\]
Here $\Sp(4)$ is the compact symplectic group of rank four,
$\Lambda^4_0\C^8$ is the kernel of symplectic contraction, and $S^+$ is
one of the two $128$-dimensional half-spin modules.  In the last row,
$\mu_2$ is the central subgroup acting trivially on the chosen $S^+$; thus
$K_8$ is the corresponding semispin group.  Interchanging $S^+$ and $S^-$
interchanges $\omega_8$ and $\omega_7$ and changes none of the statements.

The $K_i$-spectra are multiplicity free:
\begin{align*}
 \Omega_6\big|_{K_6}
    &\simeq\bigoplus_{n\geq0}V_{C_4}(n\omega_4),\\
 \Omega_7\big|_{K_7}
    &\simeq\bigoplus_{n\geq0}V_{A_7}(n\omega_4),\\
 \Omega_8\big|_{K_8}
    &\simeq\bigoplus_{n\geq0}V_{D_8}(n\omega_8).
\end{align*}
In each case the summand with $n=0$ is the trivial $K_i$-type and contains
the spherical vector.

\subsection{The Joseph ideal and its quadratic generators}

We recall the part of the construction of the Joseph ideal that will be
needed below.  All Lie algebras and modules in this paragraph are
complexified.  Let $\beta_i$ be the highest root of $\mathfrak e_i$.  There
is a multiplicity-free decomposition
\begin{equation}\label{eq:symmetric-square-adjoint}
 S^2(\mathfrak e_i)
   =V_{\mathfrak e_i}(2\beta_i)\oplus\mathbf1\oplus W_i,
\end{equation}
where $W_i$ is irreducible in the three exceptional cases under
consideration.  Its dimensions are
\[
 \dim W_6=650,\qquad \dim W_7=1539,\qquad \dim W_8=3875.
\]
More intrinsically, if $\mathfrak m_i$ is the centralizer of
$h_{\beta_i}$ and $\beta_i'$ is the highest root of
$[\mathfrak m_i,\mathfrak m_i]$, then
$W_i=V_{\mathfrak e_i}(\beta_i+\beta_i')$.

Let $C_i\in U(\mathfrak e_i)$ be the Casimir element.  The copy of $W_i$
in \eqref{eq:symmetric-square-adjoint} has a canonical lift
$\widetilde W_i\subset U_2(\mathfrak e_i)$, obtained by symmetrization.
Garfinkle's description \cite{Garfinkle}, in the form recalled
by Gan--Savin \cite{GSJ}, says that the Joseph ideal $J_i$ is generated by
\begin{equation}\label{eq:Joseph-generators}
                     \widetilde W_i\quad\text{and}\quad C_i-c_i,
\end{equation}
where $c_i$ is the Casimir eigenvalue of the minimal representation.
Let $\mathcal O_{\min}\subset\mathfrak e_i$ denote the minimal nonzero
nilpotent adjoint orbit.  Then
\begin{equation}\label{eq:minimal-orbit-ideal}
 \operatorname{gr}J_i=I(\overline{\mathcal O}_{\min})
   =\big\langle\mathbf1\oplus W_i\big\rangle\subset S(\mathfrak e_i).
\end{equation}
Equivalently, the kernel of the Cartan-product map
\[
 S^2(\mathfrak e_i)\longrightarrow V_{\mathfrak e_i}(2\beta_i)
\]
is precisely $\mathbf1\oplus W_i$.  Thus both the classical equations of
the minimal nilpotent orbit and their quantization in the Joseph ideal
are generated in degree two.

\section{The three highest-weight cones}\label{sec:cones}

Here we prove that $R_i=\operatorname{gr}\Omega_i$ is the coordinate ring
of the highest-weight cone in $\p_i$.

Let $v_i$ be a highest-weight vector in $\p_i$ and put
\[
             \widehat X_i=\overline{K_{i,\C}\cdot v_i}\subset\p_i,
             \qquad X_i=\mathbf P(\widehat X_i).
\]
The projective variety $X_i$ is the closed $K_{i,\C}$-orbit in
$\mathbf P(\p_i)$.  The three cases are
\[
\begin{array}{c|c|c}
i & X_i & \text{homogeneous embedding}\\ \hline
6 & \LG(4,8) &
 \LG(4,8)\hookrightarrow\mathbf P(\Lambda^4_0\C^8)\\
7 & \Gr(4,8) &
 \Gr(4,8)\hookrightarrow\mathbf P(\Lambda^4\C^8)\\
8 & \OGr^+(8,16) &
\OGr^+(8,16)\hookrightarrow\mathbf P(S^+).
\end{array}
\]
Their dimensions are $10$, $16$, and $28$, respectively; the corresponding
affine cones have dimensions $11$, $17$, and $29$.
The first two maps are the Pl\"ucker embeddings; in the symplectic case
the image lies in the primitive summand $\Lambda^4_0\C^8$.  The third
variety is one of the two connected components of the Grassmannian of
maximal isotropic eight-planes in $\C^{16}$.  Its embedding in
$\mathbf P(S^+)$ is the spinor, or pure-spinor, embedding.  It is important
that this is the minimal homogeneous embedding, not the ordinary Pl\"ucker
embedding into $\mathbf P(\Lambda^8\C^{16})$: the Pl\"ucker line bundle is
the square of the spinor line bundle.

Use the restriction of the Killing form to identify
$\p_i\simeq\p_i^\vee$, and put
$\varpi_i=\omega_4,\omega_4,\omega_8$, respectively, so that
$\p_i=V_{K_{i,\C}}(\varpi_i)$.  The degree-two equations of the three
cones are the kernels of the Cartan-product maps
\begin{equation}\label{eq:quadratic-Cartan-product}
 S^2(\p_i)\longrightarrow V_{K_{i,\C}}(2\varpi_i).
\end{equation}
More explicitly, one has
\begin{equation}\label{eq:quadratic-cone-table}
\begin{array}{c|c|c}
i&S^2(\p_i)&I_2(\widehat X_i)\\ \hline
6&V_{C_4}(2\omega_4)\oplus V_{C_4}(2\omega_2)\oplus\mathbf1
  &V_{C_4}(2\omega_2)\oplus\mathbf1\\
7&V_{A_7}(2\omega_4)\oplus V_{A_7}(\omega_2+\omega_6)\oplus\mathbf1
  &V_{A_7}(\omega_2+\omega_6)\oplus\mathbf1\\
8&V_{D_8}(2\omega_8)\oplus V_{D_8}(\omega_4)\oplus\mathbf1
  &V_{D_8}(\omega_4)\oplus\mathbf1.
\end{array}
\end{equation}
The dimensions of the nontrivial quadratic-relation modules in the last
column are $308$, $720$, and $1820$.  These are the symplectic Pl\"ucker,
ordinary Pl\"ucker, and pure-spinor quadrics, respectively.  Together
with the invariant quadratic equation in each row, they generate the
homogeneous ideal of $\widehat X_i$.

We now relate these equations directly to the Joseph ideal.  The Cartan
decomposition gives a $K_{i,\C}$-equivariant projection
\begin{equation}\label{eq:relative-quadratic-projection}
 \operatorname{pr}_{\p_i}:S^2(\mathfrak e_i)\longrightarrow S^2(\p_i).
\end{equation}
Under the Killing-form identifications, this projection is the restriction
of quadratic polynomials from $\mathfrak e_i$ to $\p_i$.

\begin{proposition}\label{prop:Joseph-to-Plucker}
For $i=6,7,8$,
\begin{equation}\label{eq:Joseph-to-Plucker}
 \operatorname{pr}_{\p_i}
 \ker\!\left(S^2(\mathfrak e_i)\longrightarrow
                  V_{\mathfrak e_i}(2\beta_i)\right)
   =
 \ker\!\left(S^2(\p_i)\longrightarrow
                  V_{K_{i,\C}}(2\varpi_i)\right).
\end{equation}
The restrictions of $\mathbf1\oplus W_i$ generate the ideal
$I(\widehat X_i)$.  In particular, the scheme-theoretic intersection
$\overline{\mathcal O}_{\min}\cap\p_i$ is reduced and equals
$\widehat X_i$.
\end{proposition}

\begin{proof}
Put $Q_i=\operatorname{pr}_{\p_i}(\mathbf1\oplus W_i)$.  Since
$\mathbf1\oplus W_i$ generates the ideal of the minimal orbit closure
by \eqref{eq:minimal-orbit-ideal}, the common zero set of $Q_i$ is
\[
             \mathcal V(Q_i)
               =\overline{\mathcal O}_{\min}\cap\p_i
               =\widehat X_i.
\]
Here the last equality is the geometric identification of the minimal
orbit's intersection with $\p_i$ as the highest-weight orbit, with the
origin added; see \cite[Section~3]{Br}.  At this stage we only use this
identity as an equality of closed sets.

By \eqref{eq:quadratic-cone-table},
\[
               Q_i\subset I_2(\widehat X_i)=\C q_i\oplus E_i,
\]
where $q_i$ is the restricted Killing form and $E_i$ is the nontrivial
irreducible module of dimension $308$, $720$, or $1820$, respectively.
The restriction of the invariant quadratic form on $\mathfrak e_i$ to
$\p_i$ is nonzero, so $\C q_i\subset Q_i$.  Since $Q_i$ is
$K_{i,\C}$-stable and $E_i$ is irreducible, either $Q_i=\C q_i$ or
$Q_i=\C q_i\oplus E_i$.  The first possibility would make
$\mathcal V(Q_i)$ a quadric hypersurface in $\p_i$, whereas
$\widehat X_i$ has codimension greater than one in all three cases.
Thus $Q_i=I_2(\widehat X_i)$, proving
\eqref{eq:Joseph-to-Plucker}.

Finally, $I(\widehat X_i)$ is generated by its degree-two part.  Hence
the ideal generated by the restrictions of $\mathbf1\oplus W_i$ is
exactly $I(\widehat X_i)$, which also proves the assertion about the
scheme-theoretic intersection.
\end{proof}

By the Borel--Weil Theorem, the homogeneous coordinate ring of each cone is
\begin{equation}\label{eq:cone-ring}
       \C[\widehat X_i]
       =\bigoplus_{n\geq0}H^0(X_i,\mathcal O_{X_i}(n))
       \simeq\bigoplus_{n\geq0}V_{K_{i,\C}}(n\varpi_i)^\vee.
\end{equation}
These representations are self-dual, so the dual signs may
be suppressed when one records only the $K_i$-spectrum.

Let $v_0\in\Omega_i$ be a spherical vector and give $\Omega_i$ its cyclic
PBW filtration
\[
                   F_n\Omega_i=U_n(\mathfrak e_i)v_0.
\]
Since $\mathfrak k_i v_0=0$, commuting every $\mathfrak k_i$-factor to
the right shows that the associated graded module is cyclic over
$S(\p_i)$.  Thus there is a natural surjection
\begin{equation}\label{eq:symmetric-algebra-to-graded-minimal}
 S(\p_i)\longrightarrow\operatorname{gr}_F\Omega_i,
 \qquad 1\longmapsto\operatorname{gr}(v_0).
\end{equation}
If $z\in U_2(\mathfrak e_i)$ annihilates $\Omega_i$, its leading action
on the spherical vector is the projection to $S^2(\p_i)$ of its ordinary
PBW symbol.  Indeed, this is simply the quotient
\[
 S(\mathfrak e_i)\longrightarrow
 S(\mathfrak e_i)/S(\mathfrak e_i)\mathfrak k_i\simeq S(\p_i).
\]
The constant $c_i$ in $C_i-c_i$ disappears on passing to the associated
graded.  Proposition~\ref{prop:Joseph-to-Plucker} and
\eqref{eq:Joseph-generators} therefore show that every element of
$I_2(\widehat X_i)$ annihilates the cyclic vector in
$\operatorname{gr}_F\Omega_i$.  Since the ideal of $\widehat X_i$ is
generated by these quadrics, \eqref{eq:symmetric-algebra-to-graded-minimal}
factors through a graded surjection
\begin{equation}\label{eq:cone-ring-surjects-graded}
                 \C[\widehat X_i]\twoheadrightarrow
                 \operatorname{gr}_F\Omega_i.
\end{equation}

The $K_i$-spectrum of an associated graded module agrees with that of the
filtered module.  By the Brylinski--Kostant $K_i$-type formula in the
preceding section and the Borel--Weil formula \eqref{eq:cone-ring}, the two
sides of \eqref{eq:cone-ring-surjects-graded} contain precisely the same
irreducible $K_i$-modules, each with multiplicity one.  The surjection is
therefore an isomorphism.  We have proved, without assuming a geometric
model for the associated graded module, that
\begin{equation}\label{eq:associated-graded-minimal}
                    \operatorname{gr}_F\Omega_i
                    \simeq\C[\widehat X_i].
\end{equation}
In particular, $V_{K_i}(n\varpi_i)$ first occurs in filtration degree
$n$, so the cyclic PBW filtration agrees with the $K_i$-type filtration
\[
       (\Omega_i)_{\leq n}
       =\bigoplus_{0\leq j\leq n}V_{K_i}(j\varpi_i).
\]
This also proves
$\p_i(\Omega_i)_{\leq n}\subset(\Omega_i)_{\leq n+1}$ and gives the
claimed $S(\p_i)$-action intrinsically.

Since every ring in \eqref{eq:cone-ring} is generated in degree one, it is
cyclic over $\Sym(\p_i)$, with generator $1$ in degree zero.  Proposition
\ref{prop:filtered-generation} therefore gives the familiar conclusion:
each $\Omega_i$ is generated, as a $U(\mathfrak e_i)$-module, by a lift of
its trivial $K_i$-type, namely by its spherical vector.

\section{Finite generation of compact-type lifts}\label{sec:compact-lifts}

Fix one of the three Lie algebras of type $E_n$ and now suppress the index $i$.
Thus
\[
             \mathfrak e=\kalg\oplus\p,
             \qquad \Omega=\text{the spherical minimal module},
             \qquad R=\operatorname{gr}\Omega=\C[\widehat X].
\]
Let
\[
                       \g_1\oplus\g_2\subset\mathfrak e
\]
be a reductive subalgebra preserved by the Cartan involution.  Write
\[
 \g_j=\kalg_j\oplus\p_j,
 \qquad K_j\subset K,
 \qquad j=1,2,
\]
for the induced Cartan decompositions and maximal compact groups.  The two
factors commute, so $K_1$ acts on every $K_2$-invariant or
$K_2$-covariant space.

Let $v_0\in\Omega$ be the spherical vector.  We impose the following
invariant-cyclicity hypothesis: the map
\begin{equation}\label{eq:invariant-cyclicity}
 U(\g_1)\longrightarrow\Omega^{K_2},
 \qquad u\longmapsto u v_0,
\end{equation}
is a strict filtered surjection.  Here $U(\g_1)$ has the relative PBW
filtration, in which $\kalg_1$ has degree zero and $\p_1$ has degree one,
and $\Omega^{K_2}$ has the filtration induced from $\Omega$.  Strictness is
precisely what permits passage to the associated graded.  Since $v_0$ is
$K_1$-fixed, PBW reduces the cyclic associated graded source to
$\Sym(\p_1)$.  Compactness of $K_2$ gives
$\operatorname{gr}(\Omega^{K_2})\simeq
(\operatorname{gr}\Omega)^{K_2}=R^{K_2}$; moreover, under
$\operatorname{gr}\Omega\simeq R$, the symbol action of $\p$ is
multiplication by linear functions.  Thus
\eqref{eq:invariant-cyclicity} gives a surjection of graded algebras
\begin{equation}\label{eq:graded-invariant-cyclicity}
                  \Sym(\p_1)\twoheadrightarrow R^{K_2}.
\end{equation}
Equivalently, one may take \eqref{eq:graded-invariant-cyclicity} as the
graded form of the invariant-cyclicity hypothesis.  In applications it is
usually this map that is calculated explicitly.

Let $\tau$ be an irreducible $K_2$-module and define its graded covariant
space by
\begin{equation}\label{eq:covariant-module}
 M_\tau=(R\otimes\tau^\vee)_{K_2}.
\end{equation}
Because $K_2$ is compact, averaging identifies coinvariants and invariants,
and hence
\begin{equation}\label{eq:covariant-identifications}
 M_\tau\simeq(R\otimes\tau^\vee)^{K_2}
       \simeq\operatorname{Hom}_{K_2}(\tau,R).
\end{equation}
Thus $M_\tau$ is the multiplicity, or covariant, space of $\tau$ in $R$.
For the trivial representation ${\bf1}$ this reduces to
\[
                              M_{\bf1}=R^{K_2}.
\]

We prove the finiteness statement needed here by an elementary averaging
argument.  For general invariant-theoretic background, see
\cite[Chapter~1]{MFK}.

\begin{proposition}[Finiteness of covariants]\label{prop:covariant-finiteness}
For every finite-dimensional $K_2$-module $\tau$, the covariant module
$M_\tau$ is finitely generated over $R^{K_2}$.
\end{proposition}

\begin{proof}
Let $B=R\otimes\Sym(\tau^\vee)$, graded by symmetric degree, so that
$B_0=R$ and $B_1=R\otimes\tau^\vee$.  Let $I\subset B$ be the ideal
generated by $B_1^{K_2}$.  Since $B$ is a finitely generated complex
algebra, the Hilbert basis theorem implies that $I$ is generated by a
finite subset $f_1,\ldots,f_m$ of $B_1^{K_2}$.

For $f\in B_1^{K_2}$, express $f$ in this ideal and take symmetric degree
one to obtain $f=\sum_i r_i f_i$ with $r_i\in B_0=R$.  Averaging over
the compact group $K_2$, since $f$ and the $f_i$ are invariant, gives
\[
                f=\sum_i\left(\int_{K_2} k\cdot r_i\,dk\right)f_i,
\]
where $dk$ is normalized Haar measure.  The averaged coefficients belong
to $R^{K_2}$, so the $f_i$ generate $M_\tau\simeq B_1^{K_2}$ over
$R^{K_2}$.
\end{proof}

Combining Proposition~\ref{prop:covariant-finiteness} with
\eqref{eq:graded-invariant-cyclicity}, we see that $M_\tau$ is finitely
generated over $\Sym(\p_1)$.  Since this module is $K_1$-finite and $K_1$
commutes with $K_2$, its homogeneous generators may be chosen to form a
finite-dimensional graded $K_1$-module; denote one such generating module
by $F_\tau$.

Now form the compact-type lift
\begin{equation}\label{eq:compact-type-lift}
                \Theta_1(\tau)
                =(\Omega\otimes\tau^\vee)_{K_2}.
\end{equation}
It is naturally a $(\g_1,K_1)$-module.  Give it the quotient filtration
induced from the filtration of $\Omega$.  Exactness of $K_2$-coinvariants
for a compact group gives
\begin{equation}\label{eq:graded-compact-type-lift}
 \operatorname{gr}\Theta_1(\tau)
    \simeq(\operatorname{gr}\Omega\otimes\tau^\vee)_{K_2}
    =(R\otimes\tau^\vee)_{K_2}
    =M_\tau.
\end{equation}

\begin{theorem}[Finite generation of compact-type lifts]
\label{thm:finite-generation-lifts}
Assume the graded invariant-cyclicity map
\eqref{eq:graded-invariant-cyclicity} is surjective.  Then, for every
irreducible $K_2$-type $\tau$, the $(\g_1,K_1)$-module
$\Theta_1(\tau)$ is finitely generated over $U(\g_1)$.  More precisely, if
$\widetilde F_\tau\subset\Theta_1(\tau)$ is any $K_1$-stable space of
filtered lifts of $F_\tau\subset M_\tau$, then multiplication gives a
surjection
\begin{equation}\label{eq:compact-induced-surjection}
 U(\g_1)\otimes_{U(\kalg_1)}\widetilde F_\tau
                \twoheadrightarrow\Theta_1(\tau).
\end{equation}
\end{theorem}

\begin{proof}
Proposition~\ref{prop:covariant-finiteness} and
\eqref{eq:graded-invariant-cyclicity} show that the associated graded
module in \eqref{eq:graded-compact-type-lift} is generated over
$\Sym(\p_1)$ by the finite-dimensional homogeneous $K_1$-module $F_\tau$.
Proposition~\ref{prop:filtered-generation}, applied to the filtered
$(\g_1,K_1)$-module $\Theta_1(\tau)$, now gives
\eqref{eq:compact-induced-surjection}.
\end{proof}

\section{From compact-type lifts to big theta lifts}\label{sec:big-theta}

Let $\pi_2$ be an irreducible $(\g_2,K_2)$-module.  We use the algebraic
definition
\begin{equation}\label{eq:big-theta-definition}
       \Theta(\pi_2)
       =(\Omega\otimes\pi_2^\vee)_{(\g_2,K_2)}.
\end{equation}
Equivalently, $\Theta(\pi_2)\boxtimes\pi_2$ is the maximal
$\pi_2$-isotypic quotient of $\Omega$.  In particular, there is a natural
surjection
\begin{equation}\label{eq:big-theta-quotient}
            \Omega\twoheadrightarrow
            \Theta(\pi_2)\boxtimes\pi_2.
\end{equation}

Choose a Vogan-minimal $K_2$-type $\tau$ of $\pi_2$, and put
\[
                    E_\tau=\operatorname{Hom}_{K_2}(\tau,\pi_2).
\]
This is a nonzero finite-dimensional vector space.  Since taking a fixed
$K_2$-isotypic component is exact, applying
$\operatorname{Hom}_{K_2}(\tau,-)$ to
\eqref{eq:big-theta-quotient} gives a surjection of $\g_1$-modules
\begin{equation}\label{eq:minimal-type-quotient}
 \Theta_1(\tau)=\operatorname{Hom}_{K_2}(\tau,\Omega)
   \twoheadrightarrow \Theta(\pi_2)\otimes E_\tau.
\end{equation}
Choose a nonzero linear functional $E_\tau\to\C$.  Composing it with
\eqref{eq:minimal-type-quotient} gives a surjection
\begin{equation}\label{eq:compact-to-big-quotient}
                 \Theta_1(\tau)\twoheadrightarrow\Theta(\pi_2).
\end{equation}
Thus only the occurrence of $\tau$ in $\pi_2$ is needed for this argument;
choosing a Vogan-minimal type is the natural choice for applications to
theta correspondence.

Assume the graded invariant-cyclicity hypothesis
\[
                         \Sym(\p_1)\twoheadrightarrow R^{K_2}.
\]
By Theorem~\ref{thm:finite-generation-lifts}, the compact-type lift in
\eqref{eq:minimal-type-quotient} is a quotient of the induced module
\[
 U(\g_1)\otimes_{U(\kalg_1)}\widetilde F_\tau.
\]
Together with \eqref{eq:compact-to-big-quotient}, this gives
\begin{equation}\label{eq:induced-to-big-theta}
 U(\g_1)\otimes_{U(\kalg_1)}\widetilde F_\tau
       \twoheadrightarrow\Theta(\pi_2).
\end{equation}
In particular, the big theta lift is finitely generated over $U(\g_1)$.

Admissibility of $\Theta(\pi_2)$ comes from applying the compact-type argument 
in the opposite direction.  
Assume also the reverse graded invariant-cyclicity hypothesis
\begin{equation}\label{eq:reverse-graded-invariant-cyclicity}
                         \Sym(\p_2)\twoheadrightarrow R^{K_1}.
\end{equation}
For an irreducible $K_1$-type $\sigma$, put
\begin{equation}\label{eq:reverse-compact-lift}
 \Theta_2(\sigma)
   =(\Omega\otimes\sigma^\vee)_{K_1}
   \simeq\operatorname{Hom}_{K_1}(\sigma,\Omega).
\end{equation}
Theorem~\ref{thm:finite-generation-lifts}, with the indices interchanged,
shows that $\Theta_2(\sigma)$ is finitely generated over $U(\g_2)$.

The see-saw adjunction gives
\begin{equation}\label{eq:seesaw-adjunction}
 \operatorname{Hom}_{K_1}
       (\sigma,\Theta(\pi_2))^\vee
 \simeq
 \operatorname{Hom}_{(\g_2,K_2)}
       (\Theta_2(\sigma),\pi_2).
\end{equation}
Indeed, the left-hand multiplicity space before taking its dual is
\[
 \bigl(\operatorname{Hom}_{K_1}(\sigma,\Omega)
       \otimes\pi_2^\vee\bigr)_{(\g_2,K_2)},
\]
and dualizing the coinvariants gives the right-hand Hom space in
\eqref{eq:seesaw-adjunction}.

Choose a finite-dimensional $K_2$-stable generating space $D_\sigma$ for
$\Theta_2(\sigma)$.  Restriction to $D_\sigma$ gives an injection
\begin{equation}\label{eq:hom-injection}
 \operatorname{Hom}_{(\g_2,K_2)}
       (\Theta_2(\sigma),\pi_2)
 \lhook\joinrel\longrightarrow
 \operatorname{Hom}_{K_2}(D_\sigma,\pi_2).
\end{equation}
The space on the right is finite dimensional because $\pi_2$ is
admissible.  Equations \eqref{eq:seesaw-adjunction} and
\eqref{eq:hom-injection} therefore show that every $K_1$-type has finite
multiplicity in $\Theta(\pi_2)$.

We have proved the following.

\begin{theorem}[Big theta lifts]\label{thm:big-theta-finiteness}
Let $\pi_2$ be an irreducible $(\g_2,K_2)$-module.
\begin{enumerate}
\item If $\Sym(\p_1)\twoheadrightarrow R^{K_2}$, then
$\Theta(\pi_2)$ is finitely generated over $U(\g_1)$; more precisely,
\eqref{eq:induced-to-big-theta} holds for any Vogan-minimal $K_2$-type
$\tau$ of $\pi_2$.
\item If
$\Sym(\p_2)\twoheadrightarrow R^{K_1}$, then $\Theta(\pi_2)$ is an
admissible $(\g_1,K_1)$-module.
\end{enumerate}
Consequently, under the two symmetric hypotheses, every nonzero big theta
lift is a finitely generated admissible $(\g_1,K_1)$-module and hence a
Harish--Chandra module of finite length.
\end{theorem}

The two hypotheses in Theorem~\ref{thm:big-theta-finiteness} are
symmetric, but they are not both available for all the exceptional dual
pairs considered below.  The point needed for the main theorem is that a
single invariant map controls the two opposite lifting directions.

\begin{theorem}[One side controls both directions]
\label{thm:one-sided-finite-length}
Assume that the correspondence of infinitesimal characters is available
for the dual pair $\g_1\oplus\g_2$.  If the single map
\begin{equation}\label{eq:one-sided-map}
                  \Sym(\p_1)\twoheadrightarrow R^{K_2}
\end{equation}
is surjective, then the big theta lift of every irreducible
Harish--Chandra module has finite length in both directions.
\end{theorem}

\begin{proof}
Let $\pi_2$ be an irreducible $(\g_2,K_2)$-module.  In the direction
from $\g_2$ to $\g_1$, \eqref{eq:one-sided-map} is the hypothesis of
Theorem~\ref{thm:big-theta-finiteness}(1), so $\Theta(\pi_2)$ is finitely
generated over $U(\g_1)$.  By the infinitesimal-character
correspondence it has an infinitesimal character.  The standard
Harish--Chandra finiteness theorem therefore makes it admissible and of
finite length; see \cite[Theorems~3.4.1 and~4.2.1]{WallachI}.

Now let $\pi_1$ be an irreducible $(\g_1,K_1)$-module and lift in the
opposite direction.  After interchanging the indices in
Theorem~\ref{thm:big-theta-finiteness}(2), the very same map
\eqref{eq:one-sided-map} is the reverse invariant map.  Hence the lift of
$\pi_1$ to $\g_2$ is admissible.  It again has an infinitesimal character.
There are only finitely many irreducible $(\g_2,K_2)$-modules with that
infinitesimal character \cite[Theorem~5.5.6]{WallachI}.  Choose one
$K_2$-type from each of them.  Admissibility bounds the sum of the
multiplicities of these finitely many types, and hence uniformly bounds
the length of every finitely generated submodule of the lift.  A
finitely generated submodule of maximal length is the whole lift.  Thus
the reverse lift also has finite length.
\end{proof}

\section{Checking invariant cyclicity for split dual pairs}
\label{sec:invariant-cyclicity}

The hypothesis in Theorem~\ref{thm:one-sided-finite-length} is usually
most accessible when $\g_1$ is the smaller member of the dual pair.  Let
$v_0$ again denote the spherical vector in $\Omega$.  The action map
factors through the universal spherical module
\begin{equation}\label{eq:spherical-action-map}
 \beta_1:\mathcal I_1
   :=U(\g_1)\otimes_{U(\kalg_1)}{\bf1}
   \longrightarrow \Omega^{K_2},
 \qquad u\otimes1\longmapsto u v_0.
\end{equation}
Give $\mathcal I_1$ the relative PBW filtration.  PBW identifies it, as a
filtered vector space, with $\Sym(\p_1)$.  Compactness of $K_2$ gives
\[
 \operatorname{gr}(\Omega^{K_2})
       =R^{K_2},
\]
and the associated graded map of \eqref{eq:spherical-action-map} is the
graded algebra homomorphism
\begin{equation}\label{eq:alpha-one}
 \alpha_1:\Sym(\p_1)\longrightarrow R^{K_2},
 \qquad f\longmapsto f\cdot1,
\end{equation}
where $1\in R_0$ is the symbol of $v_0$.

The injectivity required below is supplied by the following analytic
input.  It is here that matrix-coefficient decay enters the proof.

\begin{proposition}[Restriction below $L^2$ and injectivity]
\label{prop:smaller-side-injectivity}
Let $G_1$ be the smaller member of a split dual pair in
\eqref{eq:centralizer-table}, excluding the symmetric $D_4\times D_4$
pair in $E_8$.  There is an $\varepsilon>0$ such that every
$K_H$-finite matrix coefficient $M$ of the minimal representation of
$H$ satisfies
\begin{equation}\label{eq:below-L2}
                         M|_{G_1}\in L^{2-\varepsilon}(G_1).
\end{equation}
Consequently the filtered action map $\beta_1$ in
\eqref{eq:spherical-action-map} is injective.
\end{proposition}

\begin{proof}
For the trivial group both assertions are immediate.  Otherwise, the
restriction estimate is proved in
Proposition~\ref{prop:restriction-exponents} below.  We explain why it
implies injectivity.  The minimal representation is unitary, so its
matrix coefficients are bounded.  After increasing the exponent in
\eqref{eq:below-L2} if necessary, we may therefore take it strictly
between $1$ and $2$.  The matrix-coefficient pairing of
\cite[Section~3.3 and Proposition~3.2]{LS}, together with H\"older's
inequality, then gives a family of $\g_1$-intertwining maps
\[
                    m_\lambda:\Omega\longrightarrow I_1(\lambda)
\]
to spherical principal-series modules, for $\lambda$ in a neighborhood
of zero, such that $m_\lambda(v_0)$ is a nonzero multiple of the
spherical vector $v_\lambda$.

If a nonzero $X\in\Sym(\p_1)$ belonged to the kernel of $\beta_1$, then
\[
                    \pi_\lambda(X)v_\lambda=0
\]
for every $\lambda$ in that neighborhood.  The generic nonvanishing
lemma \cite[Lemma~4.1]{LS} says that
$\pi_\lambda(X)v_\lambda\neq0$ for generic $\lambda$, a contradiction.
This is precisely the argument of \cite[Corollary~4.2]{LS}; it applies
without change to the reductive groups occurring here (and is elementary
for the split torus $T_2$).  Hence $\beta_1$ is injective.
\end{proof}

\begin{proposition}[Invariant cyclicity for the smaller member]
\label{prop:filtered-dimension-criterion}
Let $G_1$ be the smaller member of a split dual pair in
\eqref{eq:centralizer-table}, excluding the symmetric $D_4\times D_4$
pair in $E_8$.  Then
\begin{equation}\label{eq:dimension-criterion}
       \dim\Sym^n(\p_1)=\dim R_n^{K_2}
       \qquad(n\geq0).
\end{equation}
Moreover, $\beta_1$ is a strict filtered isomorphism and
$\alpha_1:\Sym(\p_1)\xrightarrow{\sim}R^{K_2}$ is an isomorphism of
graded algebras.
\end{proposition}

\begin{proof}
For the $G_2$ row of \eqref{eq:centralizer-table}, the dimension identity
\eqref{eq:dimension-criterion} is \cite[Proposition~4.3]{LS}.  The
remaining dimension identities are proved below in
Proposition~\ref{prop:A1-A2-counts}, \eqref{eq:F4-row-table}, and
\eqref{eq:D4-Hilbert-series}.  Proposition
\ref{prop:smaller-side-injectivity} gives the injectivity of $\beta_1$.

The degree-$n$ filtered pieces of the source and target of $\beta_1$ have
dimensions
\[
 \sum_{j=0}^n\dim\Sym^j(\p_1)
 \quad\text{and}\quad
 \sum_{j=0}^n\dim R_j^{K_2},
\]
respectively.  They are equal by \eqref{eq:dimension-criterion}.  The
restriction of the injective map $\beta_1$ to each filtered piece is
therefore bijective.  Hence $\beta_1$ is a strict filtered isomorphism,
and passage to the associated graded gives the asserted isomorphism
$\alpha_1$.
\end{proof}

This filtered argument is the method used by Loke and Savin in the
exceptional theta correspondences in which one member is the split group
$G_{2(2)}$ \cite{LS}.  For the pairs of types
\[
       A_2\times G_2\subset E_6,
       \qquad G_2\times C_3\subset E_7,
       \qquad G_2\times F_4\subset E_8,
\]
they prove the restriction estimate and injectivity for the smaller
member and then establish \eqref{eq:dimension-criterion}; see
\cite[Corollary~4.2 and Proposition~4.3]{LS}.  Thus the smaller-side map
is a filtered isomorphism in all three cases.  Proposition
\ref{prop:smaller-side-injectivity} supplies the same injectivity input
for the additional smaller members treated below.  Loke and Savin also
prove the opposite identity in the $E_6$ and $E_7$ cases
\cite[Proposition~4.7]{LS}; no opposite identity is needed here, because
Theorem~\ref{thm:one-sided-finite-length} shows that the smaller-side
identity already suffices in both directions.

\subsection{The split centralizer table}

The same discussion can be made uniformly in the three split groups of
type $E$.  Up to finite central quotients, the connected centralizers
relevant here are
\begin{equation}\label{eq:centralizer-table}
\begin{array}{c|ccc}
G & C_{E_6}(G)^\circ&C_{E_7}(G)^\circ&C_{E_8}(G)^\circ\\ \hline
A_1&A_5&D_6&E_7\\
A_2&A_2\times A_2&A_5&E_6\\
G_2&A_2&C_3&F_4\\
D_4&T_2&A_1^3&D_4\\
F_4&\{1\}&A_1&G_2.
\end{array}
\end{equation}
Here all groups are split and $T_2\simeq(\mathbb G_m)^2$ is the split
two-dimensional torus.  Although we assumed semisimplicity in
Section~\ref{sec:compact-lifts} to simplify the notation, its proofs use
only reductivity and therefore apply without change to the
$T_2\times D_4$ pair.  The $A_1$ row records the regular $A_1$
embedding.  The $A_1$ occurring as the centralizer of $F_4$ in $E_7$ is
a different embedding and belongs to the $F_4$ row.  Finally, the two
$E_8$ entries $G_2\times F_4$ and $F_4\times G_2$ are the same unordered
pair.  Thus the fifteen entries in the table represent fourteen distinct
dual pairs.

The $G_2$ row of \eqref{eq:centralizer-table} is the result of Loke and
Savin recalled above.  Its reverse orientation is the $E_8$ member of the
$F_4$ row.  We now check the $A_1$, $A_2$, $F_4$, and $D_4$ rows; in
particular, we include the two remaining $F_4$ centralizers in $E_6$ and
$E_7$.

\subsection{The \texorpdfstring{$A_1$ and $A_2$}{A1 and A2} rows}

Lang's thesis treats the split $A_2$ cases and the corresponding complex
$A_1$ cases; see \cite[Sections~4.3--4.4 and~6.7]{Lang}.  In particular,
the $A_2$ dimension count below is \cite[Proposition~4.4.1]{Lang}.
We include the proofs for completeness.

Write $R_n=V(n\omega)$ for the degree-$n$ type in the minimal
representation.  Proposition~5.1 of \cite{LS} gives a multiplicity-free
restriction of $V(n\omega)$ to $\GL_4$, $\GL_4\times\GL_4$, and
$\widetilde{\GL}_8$ in the $E_6$, $E_7$, and $E_8$ cases, respectively.
Taking invariants of the maximal compact subgroup of the indicated
centralizer gives the following uniform count.

\begin{proposition}\label{prop:A1-A2-counts}
Let $H$ be split of type $E_6$, $E_7$, or $E_8$.
\begin{enumerate}
\item For the dual pair $A_1\times C_H(A_1)^\circ$,
\[
             \dim R_n^{K_2}=n+1=\dim\Sym^n(\p_{A_1}).
\]
\item For the dual pair $A_2\times C_H(A_2)^\circ$,
\[
             \dim R_n^{K_2}=\binom{n+4}{4}
                    =\dim\Sym^n(\p_{A_2}).
\]
Here $\dim\p_{A_1}=2$ and $\dim\p_{A_2}=5$.
\end{enumerate}
Consequently the smaller-side maps $\beta_1$ are strict filtered
isomorphisms, and
\[
 \Sym(\p_{A_1})\xrightarrow{\sim}R^{K_2},
 \qquad
 \Sym(\p_{A_2})\xrightarrow{\sim}R^{K_2}.
\]
\end{proposition}

\begin{proof}
For the $A_1$ row, a fixed vector for the maximal compact of the
centralizer forces the relevant highest weight in Proposition~5.1 of
\cite{LS} to be a determinant character.  The possible determinant
exponents are
\[
 \begin{array}{c|c}
 H&\text{parameter}\ a\\ \hline
 E_6&-n,-n+2,\ldots,n\\
 E_7&0,1,\ldots,n\\
 E_8&-n/2,-n/2+1,\ldots,n/2.
 \end{array}
\]
Thus there are $n+1$ fixed lines in each case.

For the $A_2$ row, the fixed lines are parametrized, respectively, by
\begin{align*}
 \mathcal P_6(n)&=\{n\geq\lambda_1\geq\cdots\geq\lambda_4\geq-n:
                \lambda_i\equiv n\pmod 2\},\\
 \mathcal P_7(n)&=\{n\geq\lambda_1\geq\cdots\geq\lambda_4\geq0\},\\
 \mathcal P_8(n)&=\{n/2\geq\lambda_1\geq\cdots\geq\lambda_4\geq-n/2:
                \lambda_i\equiv n/2\pmod{\mathbb Z}\}.
\end{align*}
Indeed, in the three cases the complexified compact group is,
respectively, $\SO_4$, diagonally embedded $\SL_4$, and $\Sp_8$; the
fixed-line conditions are exactly the parity, contragredience, and paired
weight conditions occurring in Proposition~5.1 of \cite{LS}.  Each set
$\mathcal P_i(n)$ is the set of multisets of size four chosen from
$n+1$ values, and hence has cardinality $\binom{n+4}{4}$.

Proposition~\ref{prop:smaller-side-injectivity} gives the injectivity of
$\beta_1$ for both $A_1$ and $A_2$.  The dimension equalities just proved
and Proposition~\ref{prop:filtered-dimension-criterion} therefore give
the two strict filtered isomorphisms.
\end{proof}

\subsection{The \texorpdfstring{$F_4$}{F4} row}

For this row it is convenient to orient the dual pair as
\[
       G_1=C_H(F_4)^\circ,
       \qquad G_2=F_4,
\]
and to let $K_{F_4}$ denote a maximal compact subgroup of the split
group $F_4$.  Thus $K_2=K_{F_4}$ in the notation of
Sections~\ref{sec:compact-lifts} and~\ref{sec:big-theta}.
The three cases are
\begin{equation}\label{eq:F4-row-table}
\begin{array}{c|c|c|c}
H&G_1&\dim\p_1&\dim R_n^{K_{F_4}}\\ \hline
E_6&\{1\}&0&\delta_{n,0}\\
E_7&\SL_2&2&n+1\\
E_8&G_2&8&\binom{n+7}{7}.
\end{array}
\end{equation}
In particular, in every case the invariant ring is the symmetric algebra
on the noncompact space of the opposite member:
\begin{equation}\label{eq:F4-row-isomorphisms}
 \begin{aligned}
 R_{E_6}^{K_{F_4}}&=\C=\Sym(0),\\
 R_{E_7}^{K_{F_4}}&\simeq\Sym(\p_{\SL_2}),\\
 R_{E_8}^{K_{F_4}}&\simeq\Sym(\p_{G_2}).
 \end{aligned}
\end{equation}
These are the graded maps induced by the action of the opposite member,
not merely abstract isomorphisms of polynomial algebras.

For completeness, the first equality says that the only
$K_{F_4}$-fixed line in the $E_6$ coordinate ring occurs in degree zero:
\[
       R_0^{K_{F_4}}=\C,
       \qquad R_n^{K_{F_4}}=0\quad(n>0).
\]
For $E_7$, write the standard eight-dimensional space as
$U\oplus W$, where $\dim U=6$, $\dim W=2$, and the complexified compact
group of $F_4$ acts through $\Sp(U)\times\SL(W)$.  Since
\[
 R_n=\mathbf S_{(n^4)}(U\oplus W),
\]
the Littlewood--Richardson rule, together with the usual symplectic
fixed-vector criterion, shows that a fixed line occurs precisely for
\[
 \lambda=(n,n,n-a,n-a),
 \qquad \mu=(a,a),
 \qquad 0\leq a\leq n.
\]
Indeed, $\lambda'$ has even parts, so
$\mathbf S_\lambda(U)$ contains one $\Sp(U)$-fixed line,
$\mathbf S_{(a,a)}(W)=\det(W)^a$ is trivial on $\SL(W)$, and
\[
 c^{(n^4)}_{(n,n,n-a,n-a),(a,a)}=1.
\]
There are no other fixed summands.  Hence
\[
                 \dim R_n^{K_{F_4}}=n+1
                    =\dim\Sym^n(\p_{\SL_2}).
\]
For the $E_7$ entry, Proposition
\ref{prop:smaller-side-injectivity} gives injectivity for the diagonal
$\SL_2$, and the equality of Hilbert functions together with Proposition
\ref{prop:filtered-dimension-criterion} gives the middle isomorphism in
\eqref{eq:F4-row-isomorphisms}.  The $E_8$ equality is exactly the
smaller-side identity for $G_2\times F_4$ recalled before
\eqref{eq:centralizer-table}; the $E_6$ entry is tautological.  This proves
\eqref{eq:F4-row-table} and \eqref{eq:F4-row-isomorphisms} in all three
groups of type $E$.

\subsection{The \texorpdfstring{$D_4$}{D4} row}

Let $K_D$ be a maximal compact subgroup of the split group of type $D_4$.
Up to a finite central quotient,
\[
                         K_D=\SU_2^4,
              \qquad \dim K_D=12.
\]
The three centralizers and the relevant dimensions are
\begin{equation}\label{eq:D4-dimension-table}
\begin{array}{c|c|c|c}
H&C_H(D_4)^\circ&\dim\p_1&\dim\widehat X\\ \hline
E_6&T_2&2&11\\
E_7&\SL_2^3&6&17\\
E_8&D_4&16&29.
\end{array}
\end{equation}
Here $\widehat X$ is the appropriate highest-weight cone from
Section~\ref{sec:cones}.

For the first two rows, the invariant calculation is again elementary.
Decompose the standard eight-dimensional space as a sum of four
symplectic planes.  The standard symplectic branching rule in the $E_6$
case, and the fourfold Littlewood--Richardson rule in the $E_7$ case,
give
\begin{equation}\label{eq:D4-Hilbert-series}
 \sum_{n\geq0}\dim R_n^{K_D}t^n
   =\begin{cases}
      (1-t)^{-2},&H=E_6,\\
      (1-t)^{-6},&H=E_7.
     \end{cases}
\end{equation}
Equivalently,
\[
 \dim R_n^{K_D}=n+1\quad(H=E_6),
 \qquad
 \dim R_n^{K_D}=\binom{n+5}{5}\quad(H=E_7).
\]
For example, in the $E_7$ calculation one writes
$R_n=\mathbf S_{(n^4)}(\C^2\oplus\C^2\oplus\C^2\oplus\C^2)$; an
$\SL_2^4$-invariant can occur only in a term indexed by four rectangular
partitions $(a_i,a_i)$, and the resulting Littlewood--Richardson sum is
\begin{equation}\label{eq:D4-LR-sum}
 \sum_{a_1+\cdots+a_4=2n}
 c^{(n^4)}_{(a_1,a_1),(a_2,a_2),(a_3,a_3),(a_4,a_4)}
       =\binom{n+5}{5}.
\end{equation}
Proposition~\ref{prop:smaller-side-injectivity} gives injectivity for
$T_2$ and $\SL_2^3$.  Since their noncompact spaces have dimensions $2$
and $6$, respectively, \eqref{eq:D4-Hilbert-series} and Proposition
\ref{prop:filtered-dimension-criterion} give strict filtered isomorphisms
\begin{equation}\label{eq:D4-good-isomorphisms}
 \Sym(\C^2)\xrightarrow{\sim}R_{E_6}^{K_D},
 \qquad
 \Sym(\C^6)\xrightarrow{\sim}R_{E_7}^{K_D}.
\end{equation}
Here $R_{E_i}$ denotes the coordinate ring of the corresponding
$E_i$ Pl\"ucker cone.

The $E_8$ case is different: the desired surjection is impossible for
dimension reasons.  Let $K_{D,\C}\simeq\SL_2^4$ be the complexification
of $K_D$.  The Pl\"ucker cone has dimension $29$, whereas
$\dim K_{D,\C}=12$.  Since this cone is factorial and $K_{D,\C}$ is
semisimple, its field of rational invariants is the fraction field of
$R^{K_D}$.  The dimension of a generic orbit is at most $12$, so
\[
       \dim R^{K_D}
       =\operatorname{trdeg}\C(\widehat X)^{K_{D,\C}}
       \geq29-12=17.
\]
On the other hand, $\dim\p_1=16$ for the other split $D_4$.  A quotient
of $\Sym(\p_1)$ has Krull dimension at most $16$, and hence there can be
no surjection
\begin{equation}\label{eq:D4-E8-failure}
                  \Sym(\p_1)\twoheadrightarrow R^{K_D}
       \qquad(D_4\times D_4\subset E_8).
\end{equation}
Thus Theorem~\ref{thm:one-sided-finite-length} applies to the $D_4$ row
for $E_6$ and $E_7$, but its invariant-surjection hypothesis cannot hold
for the symmetric $D_4\times D_4$ pair in $E_8$.

Combining these calculations with the known correspondences of
infinitesimal characters completes the proof of the main result.

\begin{proof}[Proof of Theorem~\ref{thm:main}]
The graded dimension calculations above, together with Propositions
\ref{prop:smaller-side-injectivity} and
\ref{prop:filtered-dimension-criterion}, verify the one-sided
invariant-surjection hypothesis of
Theorem~\ref{thm:one-sided-finite-length} for the $A_1$ and $A_2$ rows in
all three groups of type $E$, for the $G_2$ row, for the entire $F_4$
row, and for the $D_4$ row in $E_6$ and $E_7$.

The required matching of infinitesimal characters is given by Li
\cite{LiIC} for the pairs appearing in his table.  The one relevant
omission from that table is
\[
                    D_4\times\SL_2^3\subset E_7.
\]
It is supplied by Loke's earlier calculation \cite{Loke}.  In the
quaternionic group $E_{7,4}$, Loke considers the dual pair
\[
                 \Spin(4,4)\times\SU_2^3
\]
and determines its discrete correspondence.  His calculation of the
parameters gives the corresponding matching of infinitesimal characters.
Since the homomorphism between the centers of the enveloping algebras is
determined by the complexified dual-pair embedding, this is precisely the
missing infinitesimal-character relation for the split
$D_4\times\SL_2^3$ pair.  In fact, several of the entries recorded later
in Li's table can already be read from the correspondences in
\cite{Loke}.  Theorem~\ref{thm:one-sided-finite-length} therefore applies
to all the nondegenerate cases listed above.

The degenerate pair $\{1\}\times F_4\subset E_6$ may be handled also in the same way. 
Here $R^{K_2}=\C$ and the enveloping algebra of $F_4$ acts by an infinitesimal character. 
This follows from quaternionic groups, that is, the restriction of the minimal representation 
of the quaternionic $E_{6,4}$ to $F_4$ is a direct sum of two irreducible representations 
with the same infinitesimal character \cite{LokeTh}. 

The only remaining pair is $D_4\times D_4\subset E_8$.
Equation~\eqref{eq:D4-E8-failure} shows that the invariant-surjection
method used in this section cannot apply to it.  In the $p$-adic case,
Weissman \cite{Weissman} exhibits irreducible quotients of the form
$\pi\boxtimes\pi$ of the minimal representation with infinite
multiplicity.  Consequently, the corresponding big theta lifts are not
of finite length.  We expect the same phenomenon for the split real
pair, but do not prove it here.  As observed after
\eqref{eq:centralizer-table}, the table contains fourteen distinct
unordered pairs.  Removing this one symmetric pair leaves the thirteen
pairs asserted in the theorem.
\end{proof}

\section{Restriction estimates and injectivity}\label{sec:decay}

It remains to prove the restriction assertion used in Proposition
\ref{prop:smaller-side-injectivity}.  This is an essential step: the
estimates below give the injectivity of $\beta_1$, which combines with
the invariant counts of Section~\ref{sec:invariant-cyclicity} to produce
the required one-sided isomorphisms.  Let $M$ be a $K_H$-finite matrix
coefficient of the minimal representation of the split group $H$ of type
$E_6$, $E_7$, or $E_8$.  We call $p_S$ a critical exponent for a split
subgroup $S\subset H$ if
\[
                         M|_S\in L^{p_S+\varepsilon}(S)
                         \qquad(\varepsilon>0).
\]
As usual, no assertion at the endpoint is intended.  The pointwise
estimates underlying \cite[Section~3.2]{LS} give the following values.

\begin{proposition}\label{prop:restriction-exponents}
The following restriction estimates hold for the smaller members of the
split dual pairs in \eqref{eq:centralizer-table}, excluding the symmetric
$D_4\times D_4$ pair in $E_8$.
\begin{enumerate}
\item For the long-root $A_2\subset G_2$ in the chain
\eqref{eq:intro-chain} and the highest-root $A_1\subset A_2$, one may
take the following critical exponents $p_S$:
\begin{equation}\label{eq:A1-A2-exponents}
\begin{array}{c|ccc}
S\backslash H&E_6&E_7&E_8\\ \hline
A_1\text{ (highest root in }A_2\text{)}&4/6&3/6&2/6\\
A_2\text{ (long root in }G_2\text{)}&4/3&1&2/3.
\end{array}
\end{equation}
\item For the $G_2$ row, the smaller members have critical exponents
\begin{equation}\label{eq:G2-row-exponents}
 p_{A_2^c}=1\quad(H=E_6),
 \qquad p_{G_2}=\frac32\quad(H=E_7),
 \qquad p_{G_2}=1\quad(H=E_8),
\end{equation}
where $A_2^c=C_{E_6}(G_2)^\circ$.
\item For the two additional subgroups of $C_3\simeq\Sp_6$ in $E_7$,
one may take
\begin{equation}\label{eq:C3-subgroup-exponents}
 p_{A_1^\Delta}=2/6,
 \qquad
 p_{A_1^3}=1.
\end{equation}
Here $A_1^\Delta$ is the diagonal subgroup of the three mutually
orthogonal long-root $A_1$'s; its centralizer in $E_7$ is $F_4$, while
the centralizer of $A_1^3$ is $D_4$.
\item For $T_2=C_{E_6}(D_4)^\circ$, the restriction $M|_{T_2}$ belongs
to $L^q(T_2)$ for every $q>0$.
\end{enumerate}
In particular, for every nontrivial smaller member used in the proof of
Theorem~\ref{thm:main}, one can choose an exponent strictly below $2$.
Thus \eqref{eq:below-L2} holds for some $\varepsilon>0$.
\end{proposition}

\begin{proof}
It is convenient to separate the pointwise decay parameter from the
critical integrability exponent.  Put $q_S=p_S/2$.  Thus a bound of the
form
\begin{equation}\label{eq:rho-q-decay}
 |M(a)|\ll (1+\|\log a\|)^d a^{-\rho_S/q_S}
\end{equation}
on a positive chamber implies $M|_S\in L^{p_S+\varepsilon}(S)$ for
every $\varepsilon>0$, by the Cartan integration formula; see
\cite[Section~2.4]{WallachI}.  Indeed, the radial Haar density is bounded
by $a^{2\rho_S}$, while raising \eqref{eq:rho-q-decay} to the power
$p_S+\varepsilon$ leaves, after multiplication by this density, the
exponential factor $a^{-2\varepsilon\rho_S/p_S}$.  For the semisimple
groups considered here this dominates the polynomial factor on the
positive chamber.  The $K_H$-finiteness makes the estimate uniform over
the compact factors; the torus case is treated separately below.
The pointwise estimates for the restriction to
$G_2$ in \cite[Section~3.2]{LS} have this form, with
\[
 p_{G_2}=2,\frac32,1
 \quad\text{and hence}\quad
 q_{G_2}=1,\frac34,\frac12
\]
for $H=E_6,E_7,E_8$, respectively.

The estimates in \eqref{eq:G2-row-exponents} are also exactly those of
\cite[Section~3.2]{LS}.  In $E_6$ the smaller member of the
$A_2^c\times G_2$ pair is $A_2^c$, for which the first family of marked
diagrams gives $p_{A_2^c}=1$.  In $E_7$ and $E_8$ the smaller member is
$G_2$, and the second family gives, respectively, $3/2$ and $1$.

Let $\alpha$ and $\gamma$ be the short and long simple roots of $G_2$.
The dominant cone of the long-root $A_2\subset G_2$ is the union of two
adjacent $G_2$ cones.  On these two cones the relevant half sums are
\[
 \rho_1=5\alpha+3\gamma,
 \qquad
 \rho_2=4\alpha+3\gamma,
 \qquad
 \rho_{A_2}=3\alpha+2\gamma
             =\frac{\rho_1+\rho_2}{3}.
\]
The orthogonal projection of either $\rho_i$ onto the
$\rho_{A_2}$-line is $\tfrac32\rho_{A_2}$, and on the corresponding
half-cone the remaining component is nonnegative.  Consequently
$\rho_i\geq\tfrac32\rho_{A_2}$ there.  Restricting
\eqref{eq:rho-q-decay} therefore replaces $q_{G_2}$ by
$q_{A_2}=\tfrac23q_{G_2}$, and hence
\[
              p_{A_2}=\frac{2}{3}p_{G_2}
                       =\frac{4}{3},1,\frac{2}{3}.
\]

Now put $A_1$ at the highest root of this $A_2$.  On its split Cartan,
\[
                 \rho_{A_2}=2\rho_{A_1}.
\]
Thus $q_{A_1}=q_{A_2}/2$, or equivalently
\[
 p_{A_1}=\frac12p_{A_2}=\frac13p_{G_2}
          =\frac46,\frac36,\frac26.
\]

It remains to consider the two subgroups of $C_3$ in $E_7$.  The
$C_3$ pointwise estimate is $a^{-\rho_{C_3}}$, so $q_{C_3}=1$ and
$p_{C_3}=2$.  In the standard coordinates, the dominant cone of
$A_1^3$ is the positive octant, the union of the six $C_3$ cones, and
\[
 \rho_0=\rho_{A_1^3}=e_1+e_2+e_3,
 \qquad
 \rho_\sigma=3e_{\sigma(1)}+2e_{\sigma(2)}+e_{\sigma(3)}.
\]
The orthogonal projection of every permutation of $(3,2,1)$ onto the
diagonal line is $(2,2,2)=2\rho_0$.  More explicitly, on the cone
$x_{\sigma(1)}\geq x_{\sigma(2)}\geq x_{\sigma(3)}\geq0$ one has
\[
 (\rho_\sigma-2\rho_0)(H)
       =x_{\sigma(1)}-x_{\sigma(3)}\geq0.
\]
Therefore $a^{-\rho_\sigma}\leq a^{-2\rho_0}$, so
$q_{A_1^3}=1/2$ and $p_{A_1^3}=1$.

Finally, for the diagonal $A_1^\Delta\subset A_1^3$ its half sum, in
these coordinates, is
\[
             \rho_{A_1^\Delta}
                =\left(\frac13,\frac13,\frac13\right)
                =\frac13\rho_0.
\]
Restricting $a^{-2\rho_0}$ to its split Cartan gives
$a^{-6\rho_{A_1^\Delta}}$.  Hence $q_{A_1^\Delta}=1/6$ and
$p_{A_1^\Delta}=1/3=2/6$.  Compact Weyl representatives transfer the
estimates between the finitely many cones, and the strict
$\varepsilon>0$ absorbs all polynomial factors.

Finally, since $G_2\subset D_4\subset E_6$, one has
\[
       T_2=C_{E_6}(D_4)^\circ
          \subset C_{E_6}(G_2)^\circ=A_2^c.
\]
Both $T_2$ and $A_2^c$ have split rank two, so $T_2$ is a maximal split
torus of $A_2^c$.  The pointwise estimate corresponding to
$p_{A_2^c}=1$ is, on every Weyl chamber,
\[
 |M(a)|\ll(1+\|\log a\|)^d a^{-2\rho_{A_2^c}}.
\]
In logarithmic coordinates this is exponential decay in every nonzero
direction, whereas Haar measure on $T_2$ is Lebesgue measure.  It follows
that $M|_{T_2}\in L^q(T_2)$ for every $q>0$.

Every smaller-side critical exponent listed in the proposition is
strictly less than $2$.  Choose $r$ with
\[
                    \max\{1,p_S\}<r<2
\]
(and any $1<r<2$ for $T_2$).  Then $M|_S\in L^r(S)$; writing
$r=2-\varepsilon$ proves \eqref{eq:below-L2} and completes the proof of
Proposition~\ref{prop:smaller-side-injectivity}.
\end{proof}

\end{document}